\documentclass{amsart}
\newtheorem{thm}{Theorem}
\newtheorem{lem}[thm]{Lemma}
\newtheorem{cor}[thm]{Corollary}

\theoremstyle{definition}

\newtheorem{conj}[thm]{Conjecture}

\theoremstyle{remark}
\newtheorem{rem}[thm]{Remark}

\numberwithin{equation}{section}

\def\S{{\mathcal S}}

\def\U{{\mathcal U}}
\def\V{{\mathcal V}}

\def\op{\oplus}
\def\ot{\otimes}
\def\R{\mathbb{R}}

\begin{document}

\title[A semigroup identity for tropical $3\times3$ matrices]{A semigroup identity for tropical $3\times3$ matrices}

\author{Yaroslav Shitov}
\address{National Research University Higher School of Economics, 20 Myasnitskaya Ulitsa, Moscow 101000, Russia}
\email{yaroslav-shitov@yandex.ru}


\begin{abstract}
We construct a nontrivial semigroup identity satisfied by the tropical $3$-by-$3$ matrices.
\end{abstract}

\maketitle

\section{What is this about?}

The \textit{tropical semiring} is the set $\R$ of real numbers equipped with the tropical arithmetic,
that is, the operations $a\op b=\min\{a,b\}$ and $a\ot b=a+b$. The product of tropical matrices is defined
as the ordinary product over a field with $+$ and $\cdot$ replaced by the tropical operations $\op$ and $\ot$.
This paper is devoted to the study of tropical matrices from the point of view of semigroup theory, and we are
particularly interested in semigroup identities which tropical matrices satisfy. This problem has been considered
by Izhakian and Margolis~\cite{IzMar}, and it turns out to be interesting even for $2\times 2$ matrices.
Actually, Izhakian and Margolis have shown that the so called \textit{bicyclic} monoid admits a faithful
representation by tropical $2\times 2$ matrices, and they were able to give a shorter proof for the well-known
result by Adjan characterizing identities of bicyclic monoid~\cite{Adj}. No non-trivial semigroup identity has been known
to hold for matrices of larger order, and the question of existence of such identities has been left as an open problem,
see Section~4.3 of~\cite{IzMar}.
The further progress on this problem includes the papers~\cite{IzIden, Okni} providing different identities
that hold for \textit{triangular} tropical matrices and the paper~\cite{CHLS} that shows that the monoid
of $2\times2$ upper triangular tropical matrices is not finitely based. Despite a considerable amount
of attention in these and several other papers~\cite{CHLS, IzIden, IKR, IzMar, Okni}, the general version
of the problem remained open even in $3\times 3$ case, see Conjecture~6.2 in~\cite{IzIden}.

\begin{conj}\label{conj3x3ident}
The semigroup of tropical $3\times 3$ matrices satisfies a non-trivial semigroup identity.
\end{conj}

The contributions of this paper are as follows:

\noindent (1) we prove Conjecture~\ref{conj3x3ident} by constructing an
explicit non-trivial identity; 

\noindent (2) we construct identities satisfied by tropical
diagonally dominant $n\times n$ matrices, generalizing the result
proved in~\cite{IzIden, Okni} and giving a shorter proof for it.

\medskip


Our paper is structured as follows. We recall basic facts and relevant definitions
in Section~2, and we proceed in Section~3 with a deeper study of the notion of
\textit{sign-singularity} of tropical matrices. In Section~4, the concept of
a \textit{diagonally dominant} matrix is introduced, and a useful relation
between diagonal dominance and singularity is pointed out. Actually, the
main results of the first three sections are auxiliary;
the author cannot claim that they are original although he was not able
to find any particular matches in the literature. In Section~5, we prove
one of the main results by constructing a semigroup identity for diagonally
dominant matrices; this result is a generalization of the similar
result~\cite{IzIden, Okni} for triangular matrices. Our method seems to give
a much shorter proof of this result as Section~5 is actually self-contained
and does not rely on the results from previous sections. In Section~6, we
demonstrate our technique in use, and we prove Conjecture~\ref{conj3x3ident}. We analyze some recent progress on the topic in Section~7, and we point out the directions of future research in Section~8. We finalize the paper with the words of appreciation in Sections~9 and~10.

\section{Preliminaries}

We will denote the set of tropical $n$-by-$n$ matrices by $\R^{n\times n}$.
By $A_{ij}$ or $[A]_{ij}$ we denote the $(i,j)$th entry of a matrix $A$.
The tropical product of matrices $A$ and $B$ is denoted by $AB$,
and the $k$th tropical power of $A$ by $A^k$. Recall that by definition, the
$(i,j)$th entry of $A B$ equals $\min_{t=1}^n \{A_{it}+B_{tj}\}$ for any $i$
and $j$; the product operation is associative, so $\R^{n\times n}$ is a semigroup.
We say that $\R^{n\times n}$ satisfies a non-trivial \textit{semigroup identity}
if there are different words $\U(x,y)$ and $\V(x,y)$ from $\{x,y\}^*$
such that the condition $\U(A,B)=\V(A,B)$ holds for all $A,B\in\R^{n\times n}$.
A word $u\in\{x,y\}^*$ is called a \textit{subword} of $v\in\{x,y\}^*$ if there are
$w_1,w_2\in\{x,y\}^*$ such that $v=w_1uw_2$.

For any $s_1,\ldots,s_n\in\R$, we define a \textit{similarity transformation} on $\R^{n\times n}$, which
sends a matrix $C$ to the matrix with $(i,j)$th entry equal to $C_{ij}+s_i-s_j$. Subsets $S_1, S_2\subset\R^{n\times n}$
are called \textit{similar} if there is a similarity transformation sending $S_1$ to $S_2$. Clearly, every similarity transformation is a semigroup automorphism on $\R^{n\times n}$.

The \textit{tropical permanent} of a matrix $A\in\R^{n\times n}$ is defined as
$$\operatorname{perm}(A)=\min_{\sigma} \left\{A_{1,\sigma(1)}+\ldots+A_{n,\sigma(n)}\right\},$$
where $\sigma$ runs over the symmetric group on $\{1,\ldots,n\}$.
We will write $\Sigma(A)$ for the set of all permutations $\tau$ providing
the minimum for permanent, that is, satisfying $A_{1,\tau(1)}+\ldots
+A_{n,\tau(n)}=\operatorname{perm}(A)$. We say that $A$ is \textit{sign-nonsingular}
if all permutations in $\Sigma(A)$ have the same parity.

\section{Some properties of singular matrices}

One of the standard tools in tropical mathematics is based on the fact that the
tropical semiring can be thought of as the image of a field with non-Archimedean
valuation~\cite{DSS}. A related technique allows us to prove of the following theorem, which strenghthens the result is contained in Proposition 3.4 of~\cite{Mer}. We denote by $\S$ the ring of all formal sums of the form $s=\sum_{t\in\R} c_tX^{t}$ in which only finitely many $c_i\in\R$ are nonzero. Denote by $\deg s$ the \textit{degree} of $s$, that is, the minimal $t$ such that $c_t\neq0$. By $\S'$ we denote the subset of $\S$ consisting of those non-zero sums in which the coefficients $c_i$ are non-negative. Defining the operations on $\S$ as formal addition and multiplication, we note that the degree mapping is a homomorphism from $\S'$ to the tropical semiring.

\begin{thm}\label{lemtropnonsin}
Assume that the tropical product $AB$ of matrices $A$ and $B$ from $\R^{n\times n}$
is sign-nonsingular, assume also $\sigma\in\Sigma(A)$ and $\tau\in\Sigma(B)$.
Then we have $\tau\sigma\in\Sigma(AB)$ and $\operatorname{perm}(A)+\operatorname{perm}(B)=\operatorname{perm}(AB)$.
\end{thm}

\begin{proof}
\textit{Step~1.} Denote $\psi=\tau\sigma$. We have $[AB]_{i,\psi(i)}\leq A_{i,\sigma(i)}+B_{\sigma(i),\psi(i)}$ by the definition of tropical product,
so that $\operatorname{perm}(AB)\leq \sum_{i=1}^n[AB]_{i,\psi(i)}\leq \operatorname{perm}(A)+\operatorname{perm}(B)$.

\textit{Step~2.} Construct the matrices $A'$ and $B'$ over $\S$ (and over $\S'$) by setting $A'_{ij}=\xi_{ij}X^{A_{ij}}$ and $B'_{ij}=\chi_{ij}X^{B_{ij}}$.
Taking the coefficients $\xi_{ij}>0$ and $\chi_{ij}>0$ algebraically independent over the rationals, we ensure that $\deg\det A'=\operatorname{perm}(A)$
and $\deg\det B'=\operatorname{perm}(B)$. Defining the matrix $C'$ as the usual product of matrices $A'$ and $B'$, we see that $\det C'=\det A'\cdot\det B'$
because the determinant is multiplicative for matrices over rings. Taking the degrees of both sides in the previous equation, we get
$\deg\det C'=\operatorname{perm}(A)+\operatorname{perm}(B)$.

\textit{Step~3.} Denote by $\pi$ a permutation for which $C'_{1,\pi(1)}\ldots C'_{n,\pi(n)}$ has minimal possible degree $d$; assume
that $d<\operatorname{perm}(A)+\operatorname{perm}(B)$. Step~2 implies that the cancellation of degree-$d$ terms happens in the expression
$\det C'=\sum_{\nu} (-1)^\nu C'_{1,\nu(1)}\ldots C'_{n,\nu(n)}$. In other words, there should be a permutation $\pi'$ of parity different
from that of $\pi$ for which $C'_{1,\pi'(1)}\ldots C'_{n,\pi'(n)}$ has degree $d$. Since $\deg$ is a homomorphism, the matrix $AB$ can be
obtained from $C'$ by entrywise application of the degree mapping. Therefore $\mathrm{perm}(AB)=d$, and then $\pi,\pi'\in\Sigma(AB)$, so
$AB$ is sign-singular.

\textit{Step~4.} The contradiction obtained in Step~3 shows that $d\geq\operatorname{perm}(A)+\operatorname{perm}(B)$, that
is, $\mathrm{perm}(AB)\geq \operatorname{perm}(A)+\operatorname{perm}(B)$. Now the result follows from Step~1.
\end{proof}

\begin{cor}\cite[Theorem 9.4, part 2]{AGG}\label{lemtropnonsincor}
If the tropical product $AB$ of square matrices is sign-nonsingular, then both $A$ and $B$ are sign-nonsingular as well.
\end{cor}

%
%

\begin{cor}\label{corpower}
For $A\in\R^{n\times n}$, either $A^{n!}$ is sign-singular or $id\in\Sigma(A^{n!})$.
\end{cor}


\begin{cor}\label{lemnowwillbe}
Assume $A,B\in\R^{n\times n}$ and both $A^{n!}B^{n!}$ and $B^{n!}A^{n!}$ are sign-nonsingular.
Then $\left[A^{n!}B^{n!}\right]_{ii}=\left[B^{n!}A^{n!}\right]_{ii}$ for every $i$.
\end{cor}

\begin{proof}
By Theorem~\ref{lemtropnonsin}, both $A^{n!}B^{n!}$ and $B^{n!}A^{n!}$ have permanent $\operatorname{perm}(A^{n!})+\operatorname{perm}(B^{n!})$,
which is equal to $\sum_{i=1}^n\left[A^{n!}\right]_{ii}+\sum_{i=1}^n\left[B^{n!}\right]_{ii}$ by Corollary~\ref{corpower}.
The definition of tropical product implies $\left[A^{n!}B^{n!}\right]_{ii}\leq \left[A^{n!}\right]_{ii}+\left[B^{n!}\right]_{ii}$,
$\left[B^{n!}A^{n!}\right]_{ii}\leq \left[A^{n!}\right]_{ii}+\left[B^{n!}\right]_{ii}$, thus
$\left[A^{n!}B^{n!}\right]_{ii}=\left[B^{n!}A^{n!}\right]_{ii}=\left[A^{n!}\right]_{ii}+\left[B^{n!}\right]_{ii}$.
\end{proof}

Let us turn our attention to sign-singular matrices of order $3$. We will use the following characterization,
which is well-known in tropical linear algebra.

\begin{lem}\cite{AGG}\label{lem3x3sin}
Let a matrix $A\in\R^{3\times3}$ be sign-singular. Then there are matrices $P\in\R^{3\times2}$ and $Q\in\R^{2\times3}$ satisfying $PQ=A$.
\end{lem}

\begin{proof}
In terms of~\cite{AGG}, we need to show that the factor rank of $A$ cannot exceed $2$ if its determinantal rank does not exceed $2$.
This result follows from Theorem~8.3 and Corollary~8.12 of~\cite{AGG}.
\end{proof}

We finalize the section by showing how to construct identities for matrices which admit factorizations as those in Lemma~\ref{lem3x3sin}.

\begin{lem}\label{lem3x3fact}
Consider matrices $A=P_1Q_1$ and $B=P_2Q_2$, where $P_1,P_2\in\R^{(n+1)\times n}$ and $Q_1,Q_2\in\R^{n\times(n+1)}$.
If the identity $\U_n(A',B')=\V_n(A',B')$ holds for all matrices of order $n$, then $\U_n(A,AB)A=\V_n(A,AB)A$.
\end{lem}

\begin{proof}
This follows from the equalities $\U_n(A,AB)A=P_1\U_n(Q_1P_1,Q_1P_2Q_2P_1)Q_1$ and $\V_n(A,AB)A=P_1\,\V_n(Q_1P_1,Q_1P_2Q_2P_1)\,Q_1$.
\end{proof}

\section{Identities for matrices that are not diagonally dominant}

Now let $H$ be a positive real; we say
that a matrix $A\in\R^{n\times n}$ is \textit{diagonally $H$-dominant}
if the inequality $A_{ij}\geq\max\{A_{ii}, A_{jj}\}+H|A_{ii}-A_{jj}|$
holds for all $i$, $j$. We say that $A,B\in\R^{n\times n}$
are a \textit{diagonally $H$-dominant pair} if (1) $A_{ii}=B_{ii}$, for all $i$,
and (2) the matrix $A\oplus B$, whose $(i,j)$th entry equals $\min\{A_{ij},B_{ij}\}$,
is diagonally $H$-dominant.

\begin{rem}\label{remuptr}
Izhakian~\cite{IzIden} considers the tropical semiring extended by an infinite positive element $\infty$,
and defines a matrix $A$ to be upper triangular if $A_{ij}=\infty$ whenever $i>j$. We note that any pair
$(U,V)$ of upper triangular matrices satisfies $[UV]_{ii}=[VU]_{ii}$, for any $i$. In other words, the
pair $(UV,VU)$ is the limit of a sequence of pairs similar to diagonally $H$-dominant pairs, for arbitrarily large $H$.
\end{rem}

Let us prove some properties of diagonally dominant matrices.



\begin{lem}\cite{Kuhn}\label{lemassig}
Let a matrix $M\in\R^{n\times n}$ satisfy $M_{ii}=0$, for any $i\in\{1,\ldots,n\}$. Assume that for any permutation
$\sigma$ on $\{1,\ldots,n\}$, it holds that $M_{1\sigma(1)}+\ldots+M_{n\sigma(n)}\geq0$. Then $M$ is similar to a matrix
whose entries are nonnegative. 
\end{lem}

\begin{proof}
The \textit{Hungarian method}~\cite{Kuhn} allows one to find $r_i, s_j\in\R$
such that $M'_{ij}=M_{ij}+r_i+s_j$ is nonnegative for any $i,j$, and there is a permutation $\tau$ such that
$M'_{1\tau(1)}+\ldots+M'_{n\tau(n)}=0$. From the definition of $M'$ we get $\Sigma(M)=\Sigma(M')$, which implies
$M'_{11}+\ldots+M'_{nn}=0$, and we conclude that $r_i=-s_i$.
\end{proof}

\begin{lem}\label{lemdiagdomcha2}
Let $C$ be an $n$-by-$n$ matrix and $H$ a positive real. Assume that for any set
$K\subset\{1,\ldots,n\}$ and for any cyclic permutation $\sigma$ on $K$ it holds that
$$\sum_{\kappa\in K} C_{\kappa,\sigma(\kappa)}\geq |K|\max_{\kappa\in K}\{C_{\kappa\kappa}\}+H\sum_{\kappa\in K}|C_{\kappa,\kappa}-C_{\sigma(\kappa),\sigma(\kappa)}|.$$
Then $C$ is similar to a diagonally $H$-dominant matrix.
\end{lem}

\begin{proof}
Consider the matrix $D$ defined as $D_{ij}=C_{ij}-H|C_{ii}-C_{jj}|-\max\{C_{ii},C_{jj}\}$.
For any $K\subset\{1,\ldots,n\}$ and a cyclic permutation $\sigma$ on $K$, we have
$$\sum_{\kappa\in K} D_{\kappa,\sigma(\kappa)}=\sum_{\kappa\in K} C_{\kappa,\sigma(\kappa)}-
H\sum_{\kappa\in K}|C_{\kappa,\kappa}-C_{\sigma(\kappa),\sigma(\kappa)}|-\sum_{\kappa\in K}\max\{C_{\kappa,\kappa},C_{\sigma(\kappa),\sigma(\kappa)}\},$$
so that $\sum_{\kappa\in K} D_{\kappa,\sigma(\kappa)}\geq |K|\max_{\kappa\in K}\{C_{\kappa\kappa}\}-
\sum_{\kappa\in K}\max\{C_{\kappa,\kappa},C_{\sigma(\kappa),\sigma(\kappa)}\}\geq0$.

Now we see that the matrix $D$ satisfies the assumptions of Lemma~\ref{lemassig},
so there exist $r_1,\ldots,r_n\in\R$ such that the numbers $D_{ij}-r_i+r_j$ are nonnegative for all $i$, $j$. So the number $C'_{ij}=C_{ij}-r_i+r_j$ is not less than $H|C_{ii}-C_{jj}|+\max\{C_{ii},C_{jj}\}$,
and the matrix $C'$ is diagonally $H$-dominant.
\end{proof}

The following is a key result of the section. We enumerate by
$w_1,\ldots,w_{2^n}\in\{x,y\}^*$ all the words of length $n$ taken in an arbitrary order,
and we denote $\Gamma(x,y)=w_1\ldots w_{2^n}$.

\begin{thm}\label{lemdiagdomcha3}
Let matrices $A,B\in\R^{n\times n}$ satisfy $id\in\Sigma(A)\cap\Sigma(B)$ and $A_{ii}=B_{ii}$, for every $i$. Let $h$ be a positive integer. If $A,B$ are not similar to a diagonally $h$-dominant pair, then the matrix $\Gamma(A^{h+1}, B^{h+1})$ is sign-singular.
\end{thm}

\begin{proof}
Denote $C=A\oplus B$. Lemma~\ref{lemdiagdomcha2} shows that, for some $K\subset\{1,\ldots,n\}$ and a cyclic permutation $\sigma$ on $K$, we have
\begin{equation}\label{eqfor3}\sum_{\kappa\in K} C_{\kappa,\sigma(\kappa)}<|K|\max_{\kappa\in K}\{C_{\kappa\kappa}\}+h\sum_{\kappa\in K}|C_{\kappa,\kappa}-C_{\sigma(\kappa),\sigma(\kappa)}|.\end{equation}
We can assume that $\sigma=(k_1\,k_2\ldots k_t)$ and that $C_{k_1k_1}$ is maximal
over all $C_{kk}$ with $k\in K$. We set $X(i,j)=A$ when $A_{ij}<B_{ij}$, and $X(i,j)=B$ otherwise.
In this notation, we have $\chi_{ij}:=\left[X(i,j)^{h+1}\right]_{ij}\leq C_{ij}+h\min\{C_{ii},C_{jj}\}$. (This inequality follows from the definition of matrix multiplication.)

Denoting by $P$ the product $(X(k_1,k_2))^{h+1}\ldots (X(k_t,k_1))^{h+1}$, we obtain
$P_{k_1k_1}\leq \chi_{k_1k_2}+\ldots+\chi_{k_tk_1}$ again by the definition of matrix multiplication. We get
$$P_{k_1k_1}\leq \sum_{\kappa\in K} C_{\kappa,\sigma(\kappa)}+h\sum_{\kappa\in K}\min\{C_{\kappa,\kappa},C_{\sigma(\kappa),\sigma(\kappa)}\},$$
which implies by taking into account~\eqref{eqfor3} that $P_{k_1k_1}<(h+1)|K|C_{k_1k_1}$.
By the definition of matrix multiplication, we have $P_{ii}\leq (h+1)|K| C_{ii}$ for all $i$, so that $\operatorname{perm}(P)<(h+1)|K|\sum_{i=1}^nC_{ii}$.

If $\Gamma(A^{h+1}, B^{h+1})$ was sign-nonsingular, so would be $P$ by Corollary~\ref{lemtropnonsincor}.
Then Theorem~\ref{lemtropnonsin} would imply $\operatorname{perm}(P)=(h+1)|K|\sum_{i=1}^nC_{ii}$,
which is a contradiction.
\end{proof}

\section{Identities for diagonally dominant matrices}

In this section, we construct a semigroup identity which holds,
if $H$ is sufficiently large, for diagonally $H$-dominant matrices.
This result is a generalization of a similar result~\cite{IzIden}
for upper triangular matrices, as Remark~\ref{remuptr} shows
(because the tropical operations are continuous).

\begin{lem}\label{lemcyc}
Let $A\in\R^{n\times n}$ be a diagonally $H$-dominant matrix. Then, for any fixed index $i$,
the expression $$\alpha=A_{i i_1}+A_{i_1 i_2}+\ldots+A_{i_{h-1}i_h}+A_{i_{h} i}$$ attains its
minimum when $i_1=\ldots=i_h=i$, provided that $h+1\leq H$.
\end{lem}

\begin{proof}
Let us remove an arbitrary term $i_k$ from the sequence $(i,i_1,\ldots,i_h,i)$ if it is equal to the
preceding term. Denote the resulting sequence as $J=(j_0\ldots j_t)$; denote by $j\in J$ the index for which $A_{jj}$ is minimal. Assuming $t>0$, we get
$\alpha\geq A_{j_0 j_1}+\ldots+A_{j_{t-1} j_t}+(h-t+1)A_{jj}$. Since $A$ is $H$-dominant, we obtain
$$\alpha\geq (h+1)A_{jj}+H\sum_{\tau=0}^{t-1}\left|A_{j_\tau j_\tau}-A_{j_{\tau+1}j_{\tau+1}}\right|.$$
The triangle inequality implies $\alpha\geq(h+1)\max_{g\in J} A_{gg}\geq (h+1)A_{ii}$.
It remains to note that $\alpha=(h+1)A_{ii}$ when $i_1=\ldots=i_h=i$.
\end{proof}

In the following lemma, we denote by $G\in\{A,B\}^*$ an arbitrary word which
contains, as subwords, all the words from $\{A,B\}^*$ that have length $n$.

\begin{lem}\label{lemfordiagdom2}
Let $A,B\in\R^{n\times n}$ be a diagonally $h$-dominant pair; denote by $g$ the length of $G$ and assume $h=2ng+1$.
Choose $X^{(ng+1)}\in\{A,B\}$ arbitrarily and denote by $X^{(t)}$ the $t$th letter of the word $G^n X^{(ng+1)}G^n$.
Then, for any fixed $\kappa_0$ and $\kappa_{h}$, the expression
$$\beta=X^{(1)}_{\kappa_{0},\kappa_1}+\ldots+X^{(h)}_{\kappa_{h-1},\kappa_{h}}$$
attains its minimum on some tuple $(\kappa_1\ldots\kappa_{h-1})$ satisfying $\kappa_{ng}=\kappa_{ng+1}$.
\end{lem}

\begin{proof}
\textit{Step~1.} For $n=1$, the result is trivial; we assume $n>1$ and proceed by induction.
Let a tuple $K=(\kappa_1\ldots\kappa_{h-1})$ provide the minimum for $\beta$. By Lemma~\ref{lemcyc},
we can assume that $\kappa_p=\kappa_q$ implies $\kappa_r=\kappa_p$, provided that $p<r<q$.  Now the
consideration splits into the two cases each of which we treat separately.

\textit{Step~2.} Assume there is $u\leq g$ such that $\kappa_0\neq\kappa_u$
(or, similarly, there is $u\geq h-g$ such that $\kappa_{h}\neq\kappa_u$).
By Step~1, $\kappa_0$ does not occur among $\kappa_v$ with $v\geq g$
(in the latter case, $\kappa_{h}$ does not occur among $\kappa_v$ with $v\leq h-g$, respectively).
Let us set $c_{g}=\kappa_{g}$ and $c_{h-g}=\kappa_{h-g}$.
By induction, we can find indexes $c_{g+1},\ldots,c_{h-g-1}$ which
minimize the expression $X^{(g+1)}_{c_{g},c_{g+1}}+\ldots+X^{(h-g)}_{c_{h-g-1},c_{h-g}}$
and satisfy $c_{ng}=c_{ng+1}$. Now we are done if we change $\kappa_u$ in $K$ with $c_u$, for any $u\in\{g,\ldots,h-g\}$.

\textit{Step~3.} Now we can assume that $\kappa_i=\kappa_0$ if $i\leq g$, and that $\kappa_i=\kappa_{h}$ if $i\geq h-g$.
Denote by $e$ any index for which $A_{\kappa_e\kappa_e}$ is minimal possible; by $\{j_1,\ldots,j_t\}$ denote the set of all indexes $j$
satisfying $\kappa_{j-1}\neq\kappa_{j}$. Up to renumbering, we can assume $j_1<\ldots<j_s\leq e<j_{s+1}<\ldots<j_t$.

By convention, $G$ has a subword $X^{(j_1)}\ldots X^{(j_s)}$, so there are consecutive integers $r+1,\ldots,r+s\in\{1,\ldots,g\}$
satisfying $X^{(j_\sigma)}=X^{(r+\sigma)}$, for any $\sigma\in\{1,\ldots,s\}$. Similarly, there are consecutive integers
$q+s+1,\ldots,q+t\in\{h-g+1,\ldots,h\}$ satisfying $X^{(j_\pi)}=X^{(q+\pi)}$, for any $\pi\in\{s+1,\ldots,t\}$.
Now we set
\noindent (1) $c_i=\kappa_0$ if $i\leq r$,
\noindent (2) $c_i=\kappa_{h}$ if $i>q+t$,
\noindent (3) $c_i=\kappa_e$ if $r+s<i\leq q+s$,
\noindent (4) $c_{r+\sigma}=\kappa_{j_\sigma}$ if $\sigma\in\{1,\ldots,s\}$,
\noindent (5) $c_{q+\pi}=\kappa_{j_\pi}$ if $\pi\in\{s+1,\ldots,t\}$.
It remains to note that $c_{ng}=c_{ng+1}=\kappa_e$ and $X^{(1)}_{c_{0},c_1}+\ldots+X^{(h)}_{c_{h-1},c_{h}}\leq\beta$.
\end{proof}

Let us prove the main result of the section. Assuming that $w_1,\ldots,w_{2^n}$ is a list
of all words over the alphabet $\{A,B\}$ that have length $n$, we denote $\Gamma=w_1\ldots w_{2^n}$.

\begin{thm}\label{lemdiagdom}
If $H\geq n^2\,2^{n+1}+1$, then the identity $\Gamma^n\,A\, \Gamma^n = \Gamma^n\,B\, \Gamma^n$
holds for every pair $A,B$ of diagonally $H$-dominant $n$-by-$n$ tropical matrices.
\end{thm}

\begin{proof}
Assume $G=\Gamma$ and apply Lemma~\ref{lemfordiagdom2}. For any $\kappa_0$, $\kappa_{h}$, the quantities
$\left[\Gamma^n\,A\, \Gamma^n\right]_{\kappa_0\kappa_{h}}$ and
$\left[\Gamma^n\,B\, \Gamma^n\right]_{\kappa_0\kappa_{h}}$ are
equal to the minimum of the expression $\beta$, and, by
Lemma~\ref{lemfordiagdom2}, this minimum does not depend on $X^{(ng+1)}$.
\end{proof}

\section{The main result}

Let us apply the developed technique to construct a
non-trivial identity which holds in the semigroup of
tropical $3$-by-$3$ matrices.


\begin{thm}
Let $A$, $B$ be tropical $3\times 3$ matrices. Let $C_1=A^6B^6$, $C_2=B^6A^6$, $\Gamma(x,y)=w_1\ldots w_{8}$, where
$w_1,\ldots,w_8\in\{x,y\}^*$ are all possible words of length three. Define the words $\U(x,y)=x^2y^4x^2x^2y^2x^2y^4x^2$,
$\V(x,y)=x^2y^4x^2y^2x^2 x^2y^4x^2$, and $$H_i=\Gamma\left(C_1^{146}, C_2^{146}\right)\,\Gamma^3(C_1,C_2)\,C_i\,\Gamma^3 (C_1,C_2)$$ for $i\in\{1,2\}$.
Then $\U\left(H_1,H_1H_2\right)H_1=\V\left(H_1,H_1H_2\right)H_1$.
\end{thm}

\begin{proof}
The main result of~\cite{IzMar} states that $\U=\V$ is an identity in the semigroup of tropical $2$-by-$2$ matrices.
So if $H_1$ and $H_2$ are sign-singular, then the result follows from Lemmas~\ref{lem3x3sin} and~\ref{lem3x3fact}.
Otherwise, $C_1$, $C_2$, and $\Gamma\left(C_1^{146}, C_2^{146}\right)$ are sign-nonsingular by
Corollary~\ref{lemtropnonsincor}. From Theorem~\ref{lemtropnonsin} and Corollary~\ref{corpower} it follows that
$id\in\Sigma(C_1)\cap\Sigma(C_2)$, and Corollary~\ref{lemnowwillbe}
implies $[C_1]_{ii}=[C_2]_{ii}$ for all $i$. Now we use Theorem~\ref{lemdiagdomcha3} to conclude
that $C_1$, $C_2$ are similar to a diagonally $145$-dominant pair. From Theorem~\ref{lemdiagdom}
we deduce $H_1=H_2$, in which case the result follows because $\U(x,x)=\V(x,x)$.
\end{proof}

\section{A comment on the recent paper by Izhakian}

One of the reviewers informed me of the recent paper 'Semigroup identities of tropical matrix semigroups of maximal rank' by Zur Izhakian.\footnote{\textit{Semigroup Forum} 92(3) (2016) 712--732.} I would like to thank the reviewer for doing this, and now I am going to compare the progress achieved in Izhakian's writing with the content of the present paper. An important thing to note is that my paper appeared on arXiv\footnote{Preprint (2014) arXiv:1406.2601v1.} one year before Izhakian's paper was submitted to \textit{Semigroup Forum}. The present version differs from the first arXiv preprint by minor corrections only, so the main results of my paper are anyway original.

As the abstract and introduction of Izhakian's paper suggest, the only substantial result of it is the existence of an identity that holds in any semigroup consisting of $n\times n$ non-singular tropical matrices. This result follows from my technique immediately.

\begin{cor}\label{thrmm}
There is a non-trivial identity that holds in any semigroup consisting of sign-nonsingular tropical $n\times n$ matrices.
\end{cor}

\begin{proof}
Any matrices $X$, $Y$ in this semigroup satisfy $id \in \Sigma (X^{n!}) \cap \Sigma(Y^{n!})$ by Corollary~\ref{corpower}, and the diagonals of $A=X^{n!}Y^{n!}$ and $B=Y^{n!}X^{n!}$ coincide by Corollary~\ref{lemnowwillbe}. For any $h$, the pair $(A,B)$ is similar to a diagonally $h$-dominant pair by Theorem~\ref{lemdiagdomcha3}, so the matrices satisfy the identity as in Theorem~\ref{lemdiagdom}.
\end{proof}

We note in passing that Izahkian proves Corollary~\ref{thrmm} for the \textit{tropical non-singularity} of matrices, which is a more restrictive property than sign non-singularity.~\footnote{Recall that square matrix $A$ is \textit{tropically non-singular} if $\Sigma(A)$ is a singleton set.} This means that Corollary~\ref{thrmm} is strictly stronger than the result of Izhakian's. More than that, Corollary~\ref{thrmm} is implicitly contained in the first arXiv version of my paper, in which I write that Theorems~\ref{lemdiagdomcha3} and~\ref{lemdiagdom} \textit{'reduce the problem of constructing
an identity to sign-singular matrices'}.

The lack of originality in the substantial results of Izhakian's paper is an unpleasant circumstance, but it was even more disappointing for me to learn that the technique used by him looks very similar to my method and gives no improvement to it. In fact, Proposition~2.4 in his paper coincides\footnote{Up to the above mentioned difference between the tropical and sign singularities, which makes my results even stronger than Izhakian's.} with my Theorem~\ref{lemtropnonsin}, his Lemma~2.8 is my Corollary~\ref{corpower}, his Corollary~2.16 is my Corollary~\ref{lemnowwillbe}. These statements form an important part of the argument, and Izhakian does not mention that they appeared earlier in my paper.

Let me stress that Izhakian was aware of my paper because he have cited it. In the only mention of my paper, he writes that the \textit{'semigroup identities of tropical matrices have been ... dealt restrictively'} in my paper. I have no idea what is this supposed to mean, but these are definitely not the words that should be written about a paper that already contains the results a person is trying to prove.

Besides the facts mentioned above, it should be noted that the argument leading Izhakian to his version of Corollary~\ref{thrmm} is not valid. In particular, the first sentence of the 'proof' of Lemma~2.11 says that the author is going to prove that the graph $G_B$ is '1-cyclic reducible' while Definition~2.10 introducing the 1-cyclic reducibility of matrices says that we need to discuss a completely different graph $G_{\langle B\rangle}$ instead. The same ambiguity appears in Theorem~2.22, and its precise meaning (if any) also remains unclear. I do not immediately see how to correct the argument, and I am not sure that arising difficulties are only caused by the complicated notation.
Unfortunately, these issues were left unnoticed by the authors of the recent paper~\cite{DJK}, who give Izhakian's writing as a reference to a version of Corollary~\ref{thrmm} as if it was really his result and as if he had really proved it.

\section{Further work}

Situations like the one described in the above section appear to be very discouraging, but I hope to stay motivated enough to develop the study of this paper. A reasonable objective of further research is to write up the proof of the $n\times n$ version of Conjecture~\ref{conj3x3ident}.

\section{Acknowledgements I}

The task of publishing this paper turned out to be surprisingly hard.  I feel obliged to thank Jean-Eric Pin from \textit{Semigroup Forum}, Akihiro Munemasa from \textit{Journal of Algebraic Combinatorics}, Volodymyr Mazorchuk from \textit{Journal of Algebra}, Volker Mehrmann from \textit{Linear Algebra and its Applications}, and Elena Ivannikova from \textit{Izvestiya Mathematics} for their comments on my paper and encouraging me to keep submitting it to more suitable journals. A special thanks goes to Dijana Ilisevic from \textit{Operators and Matrices}, who informed me of the rejection of the paper because it \textit{'is not on the OaM level'}. I did always know that my paper is good, but I want to express my respect to Dijana for such a realistic assessment of the level of her journal.

\section{Acknowledgements II}

I am sincerely grateful to the editor Ted Dobson and the referees of \textit{Ars Mathematica Contemporanea} for their efforts which 
gave me a hope
to get this paper finally published. The referees' reports and editor's comments were helpful and improved the quality of the paper.

This study (research grant No 14-01-0053) was supported by The National Research University Higher School of Economics' Academic Fund Program in 2014/2015.




\begin{thebibliography}{99}

\bibitem{Adj} S.\,I. Adjan, \emph{Defining relations and algorithmic problems for groups and semigroups}, volume 85 of Proceedings of the Steklov Institute of Mathematics, Providence, 1967.

\bibitem{AGG}
{ M. Akian, S. Gaubert, A. Guterman,} Linear independence over
tropical semirings and beyond, in: G.\,L.~Litvinov, S.\,N.~Sergeev (eds.), \emph{Tropical and Idempotent Mathematics and Applications}, volume 495 of Contemporary Mathematics, 2009, 1--38.

\bibitem{CHLS}
Y. Chen, X. Hu, Y. Luo, O. Sapir, \emph{The finite basis problem for the monoid of $2\times2$ upper triangular tropical matrices}, preprint (2015), arXiv:1509.01707.

\bibitem{DJK}
L. Daviaud, M. Johnson, M. Kambites, Identities in Upper Triangular Tropical Matrix Semigroups and the Bicyclic Monoid, preprint (2016) arXiv:1612.04219v1.

\bibitem{DSS}
M.~Develin, F.~Santos, B.~Sturmfels, On the rank of a tropical matrix,
in: E.~Goodman, J.~Pach, E.~Welzl (eds.), \emph{Discrete and Computational Geometry}, volume 52 of MSRI Publications, 2005, 213--242.

\bibitem{IzIden}
Z. Izhakian, Semigroup identities in the monoid of triangular tropical matrices, \emph{Semigroup Forum} \textbf{88} (2014), 145--161.

\bibitem{IKR}
Z. Izhakian, M. Knebusch, L. Rowen, Algebraic structures of tropical mathematics, in: G.\,L.~Litvinov, S.\,N.~Sergeev (eds.), \emph{Tropical and Idempotent Mathematics and Applications}, volume 616 of Contemporary Mathematics, 2014, 125--143.

\bibitem{IzMar}
Z. Izhakian, S. W. Margolis, Semigroup identities in the monoid of two-by-two tropical matrices, \emph{Semigroup Forum} \textbf{80} (2010), 191--218.

\bibitem{KubOkn}
L. Kubat, J. Okninski, Identities of the plactic monoid, \emph{Semigroup Forum} \textbf{90} (2015), 100--112.

\bibitem{Kuhn}
H. W. Kuhn, The Hungarian Method for the assignment problem, \emph{Naval Research Logistics Quarterly} \textbf{2} (1955), 83--97.

\bibitem{Okni}
J. Okninski, Identities of the Semigroup of Upper Triangular Tropical Matrices, \emph{Communications in Algebra} \textbf{43} (2015), 4422--4426.

\bibitem{Mer}
G. Merlet, Semigroup of matrices acting on the max-plus projective space, \emph{Linear Algebra Appl.} \textbf{432} (2010), 1923--1935.
\end{thebibliography}
\end{document}